\documentclass[12pt]{amsart}%
\usepackage{amsfonts}
\usepackage{amsmath}
\usepackage{amssymb}
\usepackage{graphicx}%
\providecommand{\U}[1]{\protect\rule{.1in}{.1in}}
\newtheorem{theorem}{Theorem}[section]
\theoremstyle{plain}

\newtheorem{corollary}{Corollary}[section]

\newtheorem{lemma}{Lemma}[section]

\newtheorem{proposition}{Proposition}[section]

\numberwithin{equation}{section}
\begin{document}
\title[ ]{Sharp singular Adams inequalities in high order Sobolev spaces}
\author{Nguyen Lam}
\author{Guozhen Lu}
\address{Nguyen Lam and Guozhen Lu\\
Department of Mathematics\\
Wayne State University\\
Detroit, MI 48202, USA\\
Emails: nguyenlam@wayne.edu and gzlu@math.wayne.edu}
\thanks{Corresponding Author: G. Lu at gzlu@math.wayne.edu}
\thanks{Research is partly supported by a US NSF grant DMS0901761.}
\date{\today}
\keywords{Moser-Trudinger inequalities, Adams type inequalities, singular Adams
inequalities, fractional integrals.}
\dedicatory{ }
\begin{abstract}
In this paper, we prove a version of weighted inequalities of exponential type
for fractional integrals with sharp constants in any domain of finite measure
in $\mathbb{R}^{n}$. Using this we prove a sharp singular Adams inequality in
high order Sobolev spaces in bounded domain at critical case. Then we prove
sharp singular Adams inequalities for high order derivatives on unbounded
domains. Our results extend the singular Moser-Trudinger inequalities of first
order in \cite{Ad2, R, LR, AdY} to the higher order Sobolev spaces $W^{m,
\frac{n}{m}}$ and the results of \cite{RS} on Adams type inequalities in
unbounded domains to singular case. Our singular Adams inequality on
$W^{2,2}\left(  \mathbb{R}^{4}\right)  $ with standard Sobolev norm at the critical case settles a
unsolved question remained in \cite{Y}.

\end{abstract}
\maketitle

\section{Introduction}

Let $\Omega\subset%
\mathbb{R}
^{n}$, $n\geq2$ be a smooth bounded domain, and $W_{0}^{1,n}\left(
\Omega\right)  $ be the completion of $C_{0}^{\infty}\left(  \Omega\right)  $
under the norm $\left\Vert u\right\Vert _{W_{0}^{1,n}\left(  \Omega\right)
}=\left[  \int_{\Omega}\left(  \left\vert u\right\vert ^{2}+\left\vert \nabla
u\right\vert ^{2}\right)  dx\right]  ^{1/2}$. The classical Moser-Trudinger
inequality \cite{Mo, Po, Tru, Yu} which plays an important role in analysis
says that
\[
\underset{u\in W_{0}^{1,n}\left(  \Omega\right)  ,~\left\Vert \nabla
u\right\Vert _{n}\leq1}{\sup}\frac{1}{\left\vert \Omega\right\vert }%
\int_{\Omega}\exp\left(  \beta\left\vert u\right\vert ^{\frac{n}{n-1}}\right)
dx<+\infty
\]
for any $\beta\leq\beta_{n}=n\omega_{n-1}^{\frac{1}{n-1}}$, where
$\omega_{n-1}=\frac{2\pi^{\frac{n}{2}}}{\Gamma\left(  \frac{n}{2}\right)  }$
is the area of the surface of the unit $n-$ball. Moreover, this constant
$\beta_{n}$ is sharp in the sense that if $\beta>\beta_{n}$, then supremum is
infinity. Here and in the sequel, for any real number $p>1,$\ $\left\Vert
\cdot\right\Vert _{p}$ denotes the $L^{p}$-norm with respect to the Lebesgue measure.

There is also another famous inequality in analysis: the Hardy inequality.
Thus it is very natural to establish an interpolation of Hardy inequality and
Moser-Trudinger inequality. Inspired by the following Hardy inequality
\cite{AdCR}:%
\[
\left(  \frac{n-1}{n}\right)  ^{n}\int_{\Omega}\frac{\left\vert u\right\vert
^{n}}{\left\vert x\right\vert ^{n}\left(  \log\frac{R}{\left\vert x\right\vert
}\right)  ^{n}}dx\leq\int_{\Omega}\left\vert \nabla u\right\vert ^{n}%
\]
where $R\geq es\underset{\Omega}{\sup}\left\vert x\right\vert $, Adimurthi and
Sandeep proved in \cite{Ad2} a singular Moser-Trudinger inequality with the
sharp constant:

\bigskip\textbf{Theorem A.} \textit{Let $\Omega$ be an open and bounded set in
}$\mathbb{R}^{n}$\textit{. There exists a constant $C_{0}=C_{0}(n,\left\vert
\Omega\right\vert )>0$ such that }%
\[
\int_{\Omega}\frac{\exp\left(  \beta\left\vert u\right\vert ^{\frac{n}{n-1}%
}\right)  }{\left\vert x\right\vert ^{\alpha}}dx\leq C_{0}%
\]
\textit{for any }$\alpha\in\left[  0,n\right)  ,~0\leq\beta\leq\left(
1-\frac{\alpha}{n}\right)  \beta_{n}$\textit{, any }$u\in W_{0}^{1,n}\left(
\Omega\right)  $\textit{ with }$\int_{\Omega}\left\vert \nabla u\right\vert
^{n}dx\leq1$\textit{. Moreover, this constant }$\left(  1-\frac{\alpha}%
{n}\right)  \beta_{n}$\textit{ is sharp in the sense that if }$\beta>\left(
1-\frac{\alpha}{n}\right)  \beta_{n}$\textit{, then the above inequality can
no longer hold with some }$C_{0}$\textit{ independent of }$u$\textit{.}

There is another improved Moser-Trudinger inequality on the disk in $%
\mathbb{R}
^{2}$, which was recently proved and studied in \cite{Ad6, MS}:%
\[
\underset{u\in W_{0}^{1,2}\left(  B\right)  ,~\left\Vert \nabla u\right\Vert
_{2}\leq1}{\sup}\int_{B}\frac{\exp\left(  4\pi\left\vert u\right\vert
^{2}\right)  -1}{\left(  1-\left\vert x\right\vert ^{2}\right)  ^{2}%
}dx<+\infty.
\]
Very recently, Wang and Ye \cite{WY} proved an interesting
Hardy-Moser-Trudinger inequality on the unit disk in $%
\mathbb{R}
^{2}$, which improves the classical Moser-Trudinger inequality and the
classical Hardy inequality at the same time. Namely, there exists a constant
$C_{0}>0$ such that%
\[
\int_{B}e^{\frac{4\pi u^{2}}{H(u)}}dx\leq C_{0}<\infty,~\forall u\in
C_{0}^{\infty}\left(  B\right)  \setminus\left\{  0\right\}  ,
\]
where
\[
H(u)=\int_{B}\left\vert \nabla u\right\vert ^{2}dx-\int_{B}\frac{u^{2}%
}{\left(  1-\left\vert x\right\vert ^{2}\right)  ^{2}}dx.
\]

\bigskip We notice that when $\Omega$ has infinite volume, the usual
Moser-Trudinger inequalities become meaningless. In the case $\left\vert
\Omega\right\vert =+\infty$, the following modified Moser-Trudinger type
inequality can be established:

\bigskip\textbf{Theorem B.} \textit{For all }$\beta>0,~0\leq\alpha<n$\textit{
and }$u\in W^{1,n}\left(
\mathbb{R}
^{n}\right)  \ (n\geq2)$\textit{, there holds}%
\[
\int_{%
\mathbb{R}
^{n}}\frac{\phi\left(  \beta\left\vert u\right\vert ^{\frac{n}{n-1}}\right)
}{\left\vert x\right\vert ^{\alpha}}dx<\infty.
\]
\textit{Furthermore, we have for all }$\beta\leq\left(  1-\frac{\alpha}%
{n}\right)  \beta_{n}$\textit{ and }$\tau>0,$%
\[
\underset{\left\Vert u\right\Vert _{1,\tau}\leq1}{\sup}\int_{%
\mathbb{R}
^{n}}\frac{\phi\left(  \beta\left\vert u\right\vert ^{\frac{n}{n-1}}\right)
}{\left\vert x\right\vert ^{\alpha}}dx<\infty
\]
\textit{where }%
\begin{align*}
\phi(t)  &  =e^{t}-%
{\displaystyle\sum\limits_{j=0}^{n-2}}
\frac{t^{j}}{j!}\\
\left\Vert u\right\Vert _{1,\tau}  &  =\left(  \int_{%
\mathbb{R}
^{n}}\left(  \left\vert \nabla u\right\vert ^{n}+\tau\left\vert u\right\vert
^{n}\right)  dx\right)  ^{1/n}.
\end{align*}
\textit{Moreover, this constant }$\left(  1-\frac{\alpha}{n}\right)  \beta
_{n}$\textit{ is sharp in the sense that if }$\beta>\left(  1-\frac{\alpha}%
{n}\right)  \beta_{n}$\textit{, then the supremum is infinity.}

The above modified Moser-Trudinger type inequality when $\alpha=0$ was
established by B. Ruf \cite{R} in dimension two and Y.X. Li and Ruf \cite{LR}
in general dimension. It was then extended to the singular case $0\leq
\alpha<n$ by Adimurthi and Yang \cite{AdY}. Indeed, such type of inequality on
unbounded domains in the subcritical case $\beta<\beta_{n}$ ($\alpha=0$) was
first established by D. Cao \cite{C} in dimension two and by Adachi and Tanaka
\cite{AT} in high dimension.

In the case of compactly supported functions, D. Adams \cite{A} extended the
original Moser-Trudinger inequality to the higher order space $W_{0}%
^{m,\frac{n}{m}}(\Omega)$. In fact, Adams proved the following inequality:

\bigskip\textbf{Theorem C.} \textit{ There exists a constant} $C_{0}=C(n,m)>0$
\textit{such that for any } $u\in W_{0}^{m,\frac{n}{m}}(\Omega)$ \textit{and}
$||\nabla^{m}u||_{L^{\frac{n}{m}}(\Omega)}\leq1$, \textit{then}
\[
\frac{1}{|\Omega|}\int_{\Omega}\exp(\beta|u(x)|^{\frac{n}{n-m}})dx\leq C_{0}%
\]
\textit{for all }$\beta\leq\beta(n,m)$ \textit{where}%
\[
\mathit{\beta(n,\ m)\ =}\left\{
\begin{array}
[c]{c}%
\frac{n}{w_{n-1}}\left[  \frac{\pi^{n/2}2^{m}\Gamma(\frac{m+1}{2})}%
{\Gamma(\frac{n-m+1}{2})}\right]  ^{\frac{n}{n-m}}\text{ }\mathrm{when}%
\,\,\,m\,\,\mathrm{is\,\,odd}\\
\frac{n}{w_{n-1}}\left[  \frac{\pi^{n/2}2^{m}\Gamma(\frac{m}{2})}{\Gamma
(\frac{n-m}{2})}\right]  ^{\frac{n}{n-m}}\text{ \ }\mathrm{when}%
\,\,\,m\,\,\mathrm{is\,\,even}%
\end{array}
\right.  .
\]
\textit{Furthermore, for any } $\beta>\beta(n,m)$, \textit{the integral can be
made as large as possible.}

Note that $\beta(n,1)$ coincides with Moser's value of $\beta_{n}$ and
$\beta(2m,m)=2^{2m}\pi^{m}\Gamma(m+1)$ for both odd and even $m$. Here, we use
the symbol $\nabla^{m}u$, where $m$ is a positive integer, to denote the
$m-$th order gradient for $u\in C^{m}$, the class of $m-$th order
differentiable functions:%
\[
\nabla^{m}u=\left\{
\begin{array}
[c]{l}%
\bigtriangleup^{\frac{m}{2}}u\text{ \ \ \ \ \ \ }\mathrm{for}%
\,\,\,m\,\,\mathrm{even}\\
\nabla\bigtriangleup^{\frac{m-1}{2}}u\,\,\ \mathrm{for}\,\,\,m\,\,\mathrm{odd}%
\end{array}
\right.  .
\]
where $\nabla$ is the usual gradient operator and $\bigtriangleup$ is the
Laplacian. We use $||\nabla^{m}u||_{p}$ to denote the $L^{p}$ norm ($1\leq
p\leq\infty$) of the function $|\nabla^{m}u|$, the usual Euclidean length of
the vector $\nabla^{m}u$. We also use $W_{0}^{k,p}(\Omega)$ to denote the
Sobolev space which is a completion of $C_{0}^{\infty}(\Omega)$ under the norm
of $\left(
{\displaystyle\sum\limits_{j=0}^{k}}
||\nabla^{j}u||_{L^{p}(\Omega)}^{p}\right)  ^{1/p}$.

Recently, in the setting of the Sobolev space with homogeneous Navier boundary
conditions $W_{N}^{m,\frac{n}{m}}\left(  \Omega\right)  :$%

\[
W_{N}^{m,\frac{n}{m}}\left(  \Omega\right)  :=\left\{  u\in W^{m,\frac{n}{m}%
}:\Delta^{j}u=0\text{ on }\partial\Omega\text{ for }0\leq j\leq\left[
\frac{m-1}{2}\right]  \right\}  ,
\]
the Adams inequality was extended by Tarsi \cite{Ta}. Note that $W_{N}%
^{m,\frac{n}{m}}\left(  \Omega\right)  $ contains the Sobolev space
$W_{0}^{m,\frac{n}{m}}\left(  \Omega\right)  $\textit{ }as a closed subspace.

The Adams type inequality on Sobolev spaces $W_{0}^{m,\frac{n}{m}}\left(
\Omega\right)  $ when $\Omega$ has infinite volume and $m$ is an even integer
was studied recently by Ruf and Sani \cite{RS}. In fact, they proved the following

\bigskip\textbf{Theorem D.} \textit{If $m$ is an even integer less than} $n$,
\textit{ then there exists a constant} $C_{m,n}>0$ \textit{ such that for any
domain} $\Omega\subseteq%
\mathbb{R}
^{n}$%
\[
\underset{u\in W_{0}^{m,\frac{n}{m}}\left(  \Omega\right)  ,\left\Vert
u\right\Vert _{m,n}\leq1}{\sup}\int_{\Omega}\phi\left(  \beta_{0}\left(
n,m\right)  \left\vert u\right\vert ^{\frac{n}{n-m}}\right)  dx\leq C_{m,n}%
\]
\textit{ where}
\begin{align*}
\beta_{0}\left(  n,m\right)   &  =\frac{n}{\omega_{n-1}}\left[  \frac
{\pi^{\frac{n}{2}}2^{m}\Gamma\left(  \frac{m}{2}\right)  }{\Gamma\left(
\frac{n-m}{2}\right)  }\right]  ^{\frac{n}{n-m}},\\
\phi(t)  &  =e^{t}-%
{\displaystyle\sum\limits_{j=0}^{j_{\frac{n}{m}}-2}}
\frac{t^{j}}{j!}\\
j_{\frac{n}{m}}  &  =\min\left\{  j\in%
\mathbb{N}
:j\geq\frac{n}{m}\right\}  \geq\frac{n}{m}.
\end{align*}
\textit{Moreover, this inequality is sharp in the sense that if we replace
$\beta_{0}(n, m)$ by any larger $\beta$, then the above supremum will be
infinity. }

In the above result, Ruf and Sani used the norm
\[
\left\Vert u\right\Vert _{m,n}=\left\Vert \left(  -\Delta+I\right)  ^{\frac
{m}{2}}u\right\Vert _{\frac{n}{m}}%
\]
which is equivalent to the standard Sobolev norm
\[
\left\Vert u\right\Vert _{W^{m,\frac{n}{m}}}=\left(  \left\Vert u\right\Vert
_{\frac{n}{m}}^{\frac{n}{m}}+%
{\displaystyle\sum\limits_{j=1}^{m}}
\left\Vert \nabla^{j}u\right\Vert _{\frac{n}{m}}^{\frac{n}{m}}\right)
^{\frac{m}{n}}.
\]
In particular, if $u\in W_{0}^{m,\frac{n}{m}}\left(  \Omega\right)  $ or $u\in
W^{m,\frac{n}{m}}\left(
\mathbb{R}
^{n}\right)  $, then $\left\Vert u\right\Vert _{W^{m,\frac{n}{m}}}\leq$
$\left\Vert u\right\Vert _{m,n}$.

Because the result of Ruf and Sani \cite{RS} only treats the case when $m$ is
even, thus it leaves an open question if Ruf and Sani's theorem still holds
when $m$ is odd. Recently, the authors of \cite{LaLu4} have established the
results of Adams type inequalities on unbounded domains when $m$ is odd. More
precisely, the first result of \cite{LaLu4} is as follows:

\textbf{Theorem E.} \textit{ Let }$m$\textit{ be an odd integer less than }%
$n$\textit{: }$m=2k+1,~k\in%
\mathbb{N}
$\textit{. There holds}%
\[
\underset{u\in W^{m,\frac{n}{m}}\left(
\mathbb{R}
^{n}\right)  ,\left\Vert \nabla\left(  -\Delta+I\right)  ^{k}u\right\Vert
_{\frac{n}{m}}^{\frac{n}{m}}+\left\Vert \left(  -\Delta+I\right)
^{k}u\right\Vert _{\frac{n}{m}}^{\frac{n}{m}}\leq1}{\sup}\int_{%
\mathbb{R}
^{n}}\phi\left(  \beta\left(  n,m\right)  \left\vert u\right\vert ^{\frac
{n}{n-m}}\right)  dx<\infty.
\]
\textit{Moreover, the constant }$\beta(n,m)$\textit{ is sharp.}

In the special case $n=2m,$ we have the following stronger results in
\cite{LaLu4}:

\bigskip\textbf{Theorem F.} \textit{ Let }$m=2k+1,~k\in%
\mathbb{N}
$. \textit{For all} $\tau>0$, \textit{there holds}%
\[
\underset{u\in W^{m,2}\left(
\mathbb{R}
^{2m}\right)  ,\left\Vert \nabla\left(  -\Delta+\tau I\right)  ^{k}%
u\right\Vert _{2}^{2}+\tau\left\Vert \left(  -\Delta+\tau I\right)
^{k}u\right\Vert _{2}^{2}\leq1}{\sup}\int_{%
\mathbb{R}
^{2m}}\left(  e^{\beta(2m, m)u^{2}}-1\right)  dx<\infty.
\]
\textit{Moreover, the constant }$\beta(2m, m)$ \textit{is sharp in the sense
that if we replace} $\beta(2m, m)$ \textit{by any} $\beta>\beta(2m, m)$,
\textit{then the supremum is infinity.}

The result of \cite{RS} (stated as Theorem D above) for $m$ being even were
also extended recently using the standard Sobolev norm by Yang in the special
case $n=4$ and $m=2$ \cite{Y} and by the authors \cite{LaLu4} to the case
$n=2m$ for all $m$ being both odd and even. More precisely, the following has
been established by the authors in \cite{LaLu4}:

\bigskip\textbf{Theorem G.} \textit{Let} $m\geq2$ \textit{be an integer. For
all constants} $a_{0}=1,a_{1},...,a_{m}>0$, \textit{there holds}%
\[
\underset{u\in W^{m,2}\left(
\mathbb{R}
^{2m}\right)  ,\int_{%
\mathbb{R}
^{2m}}\left(
{\displaystyle\sum\limits_{j=0}^{m}}
a_{m-j}\left\vert \nabla^{j}u\right\vert ^{2}\right)  dx\leq1}{\sup}\int_{%
\mathbb{R}
^{2m}}\left[  \exp\left(  \beta\left(  2m,m\right)  \left\vert u\right\vert
^{2}\right)  -1\right]  dx<\infty.
\]
\textit{ Furthermore this inequality is sharp, i.e., if} $\beta(2m,m)$
\textit{is replaced by any} $\beta>\beta(2m,m)$, \textit{then the supremum is
infinite.}

As a corollary of the above theorem, we have the following Adams type
inequality with the standard Sobolev norm:

\bigskip\textbf{Theorem H.} \textit{ Let} $m\geq1$ \textit{be an integer.
There holds}%
\[
\underset{u\in W^{m,2}\left(
\mathbb{R}
^{2m}\right)  ,\left\Vert u\right\Vert _{W^{m,2}}\leq1}{\sup}\int_{%
\mathbb{R}
^{2m}}\left[  \exp\left(  \beta\left(  2m,m\right)  \left\vert u\right\vert
^{2}\right)  -1\right]  dx<\infty.
\]
\textit{Furthermore this inequality is sharp, i.e., if} $\beta(2m,m)$ \textit{
is replaced by any} $\alpha>\beta(2m,m)$, \textit{then the supremum is
infinite.}

Moser-Trudinger type inequalities and Adams type inequalities have important
applications in geometric analysis and partial differential equations,
especially in the study of the exponential growth partial differential
equations where the nonlinear term behaves like $e^{\alpha\left\vert
u\right\vert ^{\frac{n}{n-m}}}$ as $\left\vert u\right\vert \rightarrow\infty
$. There has been a vast literature in this direction. We refer the interested
reader to \cite{AP}, \cite{CC}, \cite{Ad1}, \cite{Ad2}, \cite{ASt},
\cite{AdY}, \cite{Ad4}, \cite{DeMR}, \cite{Do}, \cite{DeDoR}, \cite{LaLu1,
LaLu3} and the references therein.

In this paper, we will first establish a sharp inequality of exponential type
with weights $\frac{1}{|x|^{\alpha}}$ for the fractional integrals.

\begin{theorem}
Let $1<p<\infty$, $0\leq\alpha<n$ and $\Omega\subset\mathbb{R}^{n}$ be an open
set with $|\Omega|<\infty$. Then there is a constant $c_{0}=c_{0}(p,\Omega)$
such that for all $f\in L^{p}\left(
\mathbb{R}
^{n}\right)  $ with support contained in $\Omega$,
\[
\int_{\Omega}\frac{\exp\left(  \left(  1-\frac{\alpha}{n}\right)  \frac
{n}{\omega_{n-1}}\left\vert \frac{I_{\gamma}\ast f(x)}{\left\Vert f\right\Vert
_{p}}\right\vert ^{p^{\prime}}\right)  }{\left\vert x\right\vert ^{\alpha}%
}dx\leq c_{0},
\]
where $\gamma=n/p$ and $I_{\gamma}\ast f(x)=\int\left\vert x-y\right\vert
^{\gamma-n}f(y)dy$ is the Riesz potential of order $\gamma$.
\end{theorem}

Next, we will establish a version of singular Adams inequality on bounded
domains. More precisely, we will prove that:

\begin{theorem}
Let $0\leq\alpha<n$ and $\Omega$ be a bounded domain in $%
\mathbb{R}
^{n}$. Then for all $0\leq\beta\leq\beta_{\alpha,n,m}=\left(  1-\frac{\alpha
}{n}\right)  \beta(n,m)$, we have%
\begin{equation}
\underset{u\in W_{0}^{m,\frac{n}{m}}\left(  \Omega\right)  ,~\left\Vert
\nabla^{m} u\right\Vert _{\frac{n}{m}}\leq1}{\sup}\int_{\Omega}\frac
{e^{\beta\left\vert u\right\vert ^{\frac{n}{n-m}}}}{\left\vert x\right\vert
^{\alpha}}dx<\infty. \label{1.2}%
\end{equation}
When $\beta>\beta_{\alpha,n,m}$, the supremum is infinite. Moreover, when $m$
is an even number, the Sobolev space $W_{0}^{m,\frac{n}{m}}\left(
\Omega\right)  $ in the above supremum can be replaced by a larger Sobolev
space $W_{N}^{m,\frac{n}{m}}\left(  \Omega\right)  .$
\end{theorem}

Using the above Theorem 1.2, we will then set up the singular Adams inequality
for the space $W^{m,\frac{n}{m}}\left(
\mathbb{R}
^{n}\right)  $ when $m$ is an even integer number:

\begin{theorem}
Let $0\leq\alpha<n$, $m>0$ be an even integer less than $n$. Then for all
$0\leq\beta\leq\beta_{\alpha,n,m}=\left(  1-\frac{\alpha}{n}\right)  \beta
_{0}(n,m)$, we have%
\begin{equation}
\underset{u\in W^{m,\frac{n}{m}}\left(
\mathbb{R}
^{n}\right)  ,\left\Vert \left(  -\Delta+I\right)  ^{\frac{m}{2}}u\right\Vert
_{\frac{n}{m}}\leq1}{\sup}\int_{%
\mathbb{R}
^{n}}\frac{\phi\left(  \beta\left\vert u\right\vert ^{\frac{n}{n-m}}\right)
dx}{\left\vert x\right\vert ^{\alpha}}dx<\infty\label{1.3}%
\end{equation}
where $\phi(t) =e^{t}- \sum\limits_{j=0}^{j_{\frac{n}{m}}-2} \frac{t^{j}}{j!}%
$. Moreover, when $\beta>\beta_{\alpha,n,m}$, the supremum is infinite.
\end{theorem}

Finally, in the special case $n=2m=4,$ we will prove a singular Adams
inequality in the spirit of Theorem G above.

\begin{theorem}
\label{critical} Let $0\leq\alpha<4$. \textit{Assume that }$\tau>0$\textit{
and }$\sigma>0$\textit{ are any two positive constants. }Then for all
$0\leq\beta\leq\beta_{\alpha}=\left(  1-\frac{\alpha}{4}\right)  32\pi^{2}$,
we have%
\begin{equation}
\underset{u\in W^{2,2}\left(
\mathbb{R}
^{4}\right)  ,\int_{%
\mathbb{R}
^{4}}\left(  \left\vert \Delta u\right\vert ^{2}+\tau\left\vert \nabla
u\right\vert ^{2}+\sigma\left\vert u\right\vert ^{2}\right)  \leq1}{\sup}%
\int_{%
\mathbb{R}
^{4}}\frac{\left(  e^{\beta u^{2}}-1\right)  }{\left\vert x\right\vert
^{\alpha}}dx<\infty. \label{1.4}%
\end{equation}
Moreover, when $\beta>\beta_{\alpha}$, the supremum is infinite.
\end{theorem}

As we can see, when $\alpha=0$, this theorem is already included in Theorem G.
When $0<\alpha<4$, we note that the above inequality (\ref{1.4}) for the
subcritical case $\beta<\beta_{\alpha}=\left(  1-\frac{\alpha}{4}\right)
32\pi^{2}$ was proved in \cite{Y}. However, the critical case $\beta=\left(
1-\frac{\alpha}{4}\right)  32\pi^{2}$ is much harder to prove. Thus, our
Theorem \ref{critical} in the critical case settles a unsolved question
remained in \cite{Y}.

Our paper is organized as follows: In Section 2, we give some preliminaries.
Section 3 deals with the sharp weighted inequality of exponential type for
fractional integrals (Theorem 1.1). The singular Adams inequality for the
bounded domains (Theorem 1.2) will be proved in Section 4. Theorem 1.2 will be
used to prove Theorem 1.3 and Theorem 1.4 in Section 5.

\section{Some preliminaries}

In this section, we provide some preliminaries. For $u\in W^{m,p}\left(
\Omega\right)  $ with $1\leq p<\infty$, we will denote by $\nabla^{j}u$,
$j\in\left\{  1,2,...,m\right\}  $, the $j-th$ order gradient of $u$, namely
\[
\nabla^{j}u=\left\{
\begin{array}
[c]{l}%
\bigtriangleup^{\frac{j}{2}}u\text{ \ \ \ \ \ \ }\mathrm{for}%
\,\,\,j\,\,\mathrm{even}\\
\nabla\bigtriangleup^{\frac{j-1}{2}}u\,\,\ \mathrm{for}\,\,\,j\,\,\mathrm{odd}%
\end{array}
\right.  .
\]

We now introduce the Sobolev space of functions with homogeneous Navier
boundary conditions:%
\[
W_{N}^{m,\frac{n}{m}}\left(  \Omega\right)  :=\left\{  u\in W^{m,\frac{n}{m}%
}\left(  \Omega\right)  :\Delta^{j}u=0\text{ on }\partial\Omega\text{ for
}0\leq j\leq\left[  \frac{m-1}{2}\right]  \right\}  .
\]
It is easy to see that $W_{N}^{m,\frac{n}{m}}\left(  \Omega\right)  $ contains
$W_{0}^{m,\frac{n}{m}}\left(  \Omega\right)  $ as a closed subspace. We also
define
\begin{align*}
W_{rad}^{m,\frac{n}{m}}\left(  B_{R}\right)   &  :=\left\{  u\in W^{m,\frac
{n}{m}}(B_{R}):u(x)=u(\left\vert x\right\vert )\text{ a.e. in }B_{R}\right\}
,\\
W_{N,rad}^{m,\frac{n}{m}}\left(  B_{R}\right)   &  =W_{N}^{m,\frac{n}{m}%
}\left(  B_{R}\right)  \cap W_{rad}^{m,\frac{n}{m}}\left(  B_{R}\right)
\end{align*}
where $B_{R}=\left\{  x\in%
\mathbb{R}
^{n}:\left\vert x\right\vert <R\right\}  $ is a ball in $\mathbb{R}^{n}$.

Next, we will discuss the iterated comparison principle. Let $\Omega$ be a
bounded domain in $%
\mathbb{R}
^{n}$ and $B_{R}$ be an open ball with radius $R>0$ centered at $0$ such that
$\left\vert \Omega\right\vert =\left\vert B_{R}\right\vert $. Let
$u:\Omega\rightarrow%
\mathbb{R}
$ be a measurable function. The distribution function of $u$ is defined by
\[
\mu_{u}(t)=\left\vert \left\{  x\in\Omega|\left\vert u(x)\right\vert
>t\right\}  \right\vert ~\forall t\geq0.
\]
The decreasing rearrangement of $u$ is defined by
\[
u^{\ast}(s)=\inf\left\{  t\geq0:\mu_{u}(t)<s\right\}  \,\, ~\forall
s\in\left[  0,\left\vert B_{R}\right\vert \right]  ,
\]
and the spherically symmetric decreasing rearrangement of $u$ by%
\[
u^{\#}(x)=u^{\ast}\left(  \sigma_{n}\left\vert x\right\vert ^{n}\right)  \,\,
~\forall x\in B_{R}.
\]
We have that $u^{\#}$ is the unique nonnegative integrable function which is
radially symmetric, nonincreasing and has the same distribution function as
$\left\vert u\right\vert .$

Now, we introduce the Trombetti and Vazquez iterated comparision principle
\cite{TV}: let $c>0$ and $u$ be a weak solution of
\begin{equation}
\left\{
\begin{array}
[c]{c}%
-\Delta u+cu=f\text{ in }B_{R}\\
u\in W_{0}^{1,2}\left(  B_{R}\right)
\end{array}
\right.  \label{2.1}%
\end{equation}
where $f\in L^{\frac{2n}{n+2}}\left(  B_{R}\right)  $. We have the following
result that can be found in \cite{TV} (Inequality $(2.20)$):

\begin{proposition}
If $u$ is a nonnegative weak solution of (\ref{2.1}) then
\begin{equation}
-\frac{du^{\ast}}{ds}(s)\leq\frac{s^{\frac{2}{n}-2}}{n^{2}\sigma_{n}^{2/n}}%
{\displaystyle\int\limits_{0}^{s}}
\left(  f^{\ast}-cu^{\ast}\right)  dt,~\forall s\in\left(  0,\left\vert
B_{R}\right\vert \right)  . \label{2.2}%
\end{equation}

\end{proposition}

Now, we consider the problem
\begin{equation}
\left\{
\begin{array}
[c]{c}%
-\Delta v+cv=f^{\#}\text{ in }B_{R}\\
v\in W_{0}^{1,2}\left(  B_{R}\right)
\end{array}
\right.  \label{2.3}%
\end{equation}
Due to the radial symmetry of the equation, the unique solution $v$ of
(\ref{2.3}) is radially symmetric and we have
\begin{equation}
\label{identity}-\frac{d\widehat{v}}{ds}(s)=\frac{s^{\frac{2}{n}-2}}%
{n^{2}\sigma_{n}^{2/n}}%
{\displaystyle\int\limits_{0}^{s}}
\left(  f^{\ast}-c\widehat{v}\right)  dt,~\forall s\in\left(  0,\left\vert
B_{R}\right\vert \right)
\end{equation}
where $\widehat{v}\left(  \sigma_{n}\left\vert x\right\vert ^{n}\right)
:=v(x)$. We have the following comparison of integrals in balls that again can
be found in \cite{TV}:

\begin{proposition}
Let $u,~v$ be weak solutions of (\ref{2.1}) and (\ref{2.3}) respectively. For
every $r\in\left(  0,R\right)  $ we have%
\[
\int_{B_{r}}u^{\#}dx\leq\int_{B_{r}}vdx.
\]
and for every convex nondecreasing function $\phi:\left[  0,+\infty\right)
\rightarrow\left[  0,+\infty\right)  $ we have%
\[
\int_{B_{r}}\phi\left(  \left\vert u\right\vert \right)  dx\leq\int_{B_{r}%
}\phi\left(  \left\vert v\right\vert \right)  dx.
\]

\end{proposition}

Next, we adapt the comparison principle to the polyharmonic operator. Let
$u\in W^{m,2}\left(  B_{R}\right)  $ be a weak solution of
\begin{equation}
\left\{
\begin{array}
[c]{c}%
\left(  -\Delta+cI\right)  ^{k}u=f\text{ in }B_{R}\\
u\in W_{N}^{2k,2}\left(  B_{R}\right)
\end{array}
\right.  \label{2.4}%
\end{equation}
where $m=2k$ and $f\in L^{\frac{2n}{n+2}}\left(  B_{R}\right)  $. If we
consider the problem
\begin{equation}
\left\{
\begin{array}
[c]{c}%
\left(  -\Delta+cI\right)  ^{k}v=f^{\#}\text{ in }B_{R}\\
v\in W_{N}^{2k,2}\left(  B_{R}\right)
\end{array}
\right.  \label{2.5}%
\end{equation}
then we have the following comparison of integrals in balls:

\begin{proposition}
Let $u,~v$ be weak solutions of the polyharmonic problems (\ref{2.4}) and
(\ref{2.5}) respectively. Then for every $r\in\left(  0,R\right)  $ we have%
\[
\int_{B_{r}}u^{\#}dx\leq\int_{B_{r}}vdx.
\]

\end{proposition}

\begin{proof}
The proof adapts the comparison principle as in \cite{TV} and \cite{RS}. We
include a proof for its completeness. Since equations in (\ref{2.4}) and
(\ref{2.5}) are considered with homogeneous Navier boundary conditions, they
may be rewritten as second order systems:%

\begin{align*}
(P1)\left\{
\begin{array}
[c]{c}%
-\Delta u_{1}+cu_{1}=f\text{ in }B_{R}\\
u_{1}\in W_{0}^{1,2}\left(  B_{R}\right)
\end{array}
\right.  ~\ (Pi)\left\{
\begin{array}
[c]{c}%
-\Delta u_{i}+cu_{i}=u_{i-1}\text{ in }B_{R}\\
u_{i}\in W_{0}^{1,2}\left(  B_{R}\right)
\end{array}
\right.  ~i  &  \in\left\{  2,3,...,k\right\} \\
(Q1)\left\{
\begin{array}
[c]{c}%
-\Delta v_{1}+cv_{1}=f^{\#}\text{ in }B_{R}\\
v_{1}\in W_{0}^{1,2}\left(  B_{R}\right)
\end{array}
\right.  ~\ (Qi)\left\{
\begin{array}
[c]{c}%
-\Delta v_{i}+cv_{i}=v_{i-1}\text{ in }B_{R}\\
v_{i}\in W_{0}^{1,2}\left(  B_{R}\right)
\end{array}
\right.  ~i  &  \in\left\{  2,3,...,k\right\}
\end{align*}
where $u_{k}=u$ and $v_{k}=v$. Thus we have to prove that for every
$r\in\left(  0,R\right)  $%
\begin{equation}
\int_{B_{r}}u_{k}^{\#}dx\leq\int_{B_{r}}v_{k}dx. \label{2.6}%
\end{equation}
By the above proposition (Proposition 2.2), we have
\[
\int_{B_{r}}u_{1}^{\#}dx\leq\int_{B_{r}}v_{1}dx.
\]
Now, if we assume that
\[
\int_{B_{r}}u_{j}^{\#}dx\leq\int_{B_{r}}v_{j}dx\text{ for all }j=1,...,i,
\]
we will prove that
\[
\int_{B_{r}}u_{i+1}^{\#}dx\leq\int_{B_{r}}v_{i+1}dx.
\]
With no loss of generality, we may assume that $u_{i+1}\geq0$. In fact, let
$\overline{u}_{i+1}$ be a weak solution of
\[
\left\{
\begin{array}
[c]{c}%
-\Delta\overline{u}_{i+1}+c\overline{u}_{i+1}=\left\vert u_{i}\right\vert
\text{ in }B_{R}\\
\overline{u}_{i+1}\in W_{0}^{1,2}\left(  B_{R}\right)
\end{array}
\right.
\]
then the maximum principle implies that $\overline{u}_{i+1}\geq0$ and
$\overline{u}_{i+1}\geq\left\vert u_{i+1}\right\vert $ in $B_{R}$.

Since $u_{i+1}$ is a nonnegative weak solution of $(P\left(  i+1\right)  )$
and $v_{i+1}$ is a nonnegative weak solution of $(Q\left(  i+1\right)  )$,
then by Proposition 2.1 we have%
\begin{align*}
-\frac{du_{i+1}^{\ast}}{ds}(s)  &  \leq\frac{s^{\frac{2}{n}-2}}{n^{2}%
\sigma_{n}^{2/n}}%
{\displaystyle\int\limits_{0}^{s}}
\left(  u_{i}^{\ast}-cu_{i+1}^{\ast}\right)  dt,~\forall s\in\left(
0,\left\vert B_{R}\right\vert \right)  ,\\
-\frac{d\widehat{v}_{i+1}}{ds}(s)  &  =\frac{s^{\frac{2}{n}-2}}{n^{2}%
\sigma_{n}^{2/n}}%
{\displaystyle\int\limits_{0}^{s}}
\left(  \widehat{v}_{i}-c\widehat{v}_{i+1}\right)  dt,~\forall s\in\left(
0,\left\vert B_{R}\right\vert \right)
\end{align*}
Thus for all $s\in\left(  0,\left\vert B_{R}\right\vert \right)  $, we have
\[
\frac{d\widehat{v}_{i+1}}{ds}(s)-\frac{du_{i+1}^{\ast}}{ds}(s)-\frac
{s^{\frac{2}{n}-2}}{n^{2}\sigma_{n}^{2/n}}%
{\displaystyle\int\limits_{0}^{s}}
\left(  c\widehat{v}_{i+1}-cu_{i+1}^{\ast}\right)  dt\leq\frac{s^{\frac{2}%
{n}-2}}{n^{2}\sigma_{n}^{2/n}}%
{\displaystyle\int\limits_{0}^{s}}
\left(  u_{i}^{\ast}-\widehat{v}_{i}\right)  dt.
\]
Thanks to the induction hypotheses, we get that
\[%
{\displaystyle\int\limits_{0}^{s}}
\left(  u_{i}^{\ast}-\widehat{v}_{i}\right)  dt\leq0~,\forall s\in\left(
0,\left\vert B_{R}\right\vert \right)
\]
and then
\[
\frac{d\widehat{v}_{i+1}}{ds}(s)-\frac{du_{i+1}^{\ast}}{ds}(s)-\frac
{s^{\frac{2}{n}-2}}{n^{2}\sigma_{n}^{2/n}}%
{\displaystyle\int\limits_{0}^{s}}
\left(  c\widehat{v}_{i+1}-cu_{i+1}^{\ast}\right)  dt\leq0.
\]
Setting
\[
y(s)=%
{\displaystyle\int\limits_{0}^{s}}
\left(  \widehat{v}_{i+1}-u_{i+1}^{\ast}\right)  dt~~\forall s\in\left(
0,\left\vert B_{R}\right\vert \right)
\]
we get
\[
\left\{
\begin{array}
[c]{c}%
y^{\prime\prime}-\frac{cs^{\frac{2}{n}-2}}{n^{2}\sigma_{n}^{2/n}}%
y\leq0,~\forall s\in\left(  0,\left\vert B_{R}\right\vert \right) \\
y(0)=y^{\prime}(\left\vert B_{R}\right\vert )=0
\end{array}
\right.  .
\]
By maximum principle, we have that $y\geq0$ which is what we need.
\end{proof}

From the above proposition, we have the following corollary:

\begin{corollary}
Let $u,~v$ be weak solutions of the polyharmonic problems (\ref{2.4}) and
(\ref{2.5}) respectively. Then for every convex nondecreasing function
$\phi:\left[  0,+\infty\right)  \rightarrow\left[  0,+\infty\right)  $ we have%
\[
\int_{B_{r}}\phi\left(  \left\vert u\right\vert \right)  dx\leq\int_{B_{r}%
}\phi\left(  \left\vert v\right\vert \right)  dx.
\]

\end{corollary}

Now, we state the following known result from \cite{ALT, Ke}:

\begin{lemma}
Let $f(s),~g(s)$ be measurable, positive functions such that
\[%
{\displaystyle\int\limits_{\left[  0,r\right]  }}
f(s)ds\leq%
{\displaystyle\int\limits_{\left[  0,r\right]  }}
g(s)ds,\ r\in\left[  0,R\right]  ;
\]
if $h(s)\geq0$ is a decreasing function then
\[%
{\displaystyle\int\limits_{\left[  0,r\right]  }}
f(s)h(s)ds\leq%
{\displaystyle\int\limits_{\left[  0,r\right]  }}
g(s)h(s)ds,\ r\in\left[  0,R\right]  .
\]

\end{lemma}

Then we have the following:

\begin{proposition}
Let $u,~v$ be weak solutions of (\ref{2.4}) and (\ref{2.6}) respectively. For
every convex nondecreasing function $\phi:\left[  0,+\infty\right)
\rightarrow\left[  0,+\infty\right)  $ we have%
\[
\int_{B_{R}}\frac{\phi\left(  \left\vert u\right\vert \right)  }{\left\vert
x\right\vert ^{\alpha}}dx\leq\int_{B_{R}}\frac{\phi\left(  \left\vert
v\right\vert \right)  }{\left\vert x\right\vert ^{\alpha}}dx,~0\leq\alpha<n.
\]

\end{proposition}

Next, we provide some Radial Lemmas which will be used in the proof of Theorem
1.2. See \cite{BL, K, LaLu4, RS, Ta}:

\begin{lemma}
If $u\in W^{1,\frac{n}{m}}\left(
\mathbb{R}
^{n}\right)  $ then
\[
\left\vert u(x)\right\vert \leq\left(  \frac{1}{m\sigma_{n}}\right)
^{\frac{m}{n}}\frac{1}{\left\vert x\right\vert ^{\frac{n-1}{n}m}}\left\Vert
u\right\Vert _{W^{1,\frac{n}{m}}}%
\]
for a.e. $x\in%
\mathbb{R}
^{n}$.
\end{lemma}

\begin{lemma}
If $u\in L^{p}\left(
\mathbb{R}
^{n}\right)  ,~1\leq p<\infty,$ is a radial nonincreasing function, then
\[
\left\vert u(x)\right\vert \leq\left(  \frac{n}{\omega_{n-1}}\right)
^{\frac{1}{p}}\frac{1}{\left\vert x\right\vert ^{\frac{n}{p}}}\left\Vert
u\right\Vert _{L^{p}\left(
\mathbb{R}
^{n}\right)  }%
\]
for a.e. $x\in%
\mathbb{R}
^{n}$.
\end{lemma}

\section{Proof of Theorem 1.1: Sharp inequality of exponential type for
fractional integrals}

\bigskip We begin with proving the following result that is a modified version
of the key lemma used to prove the Adams inequality in \cite{A}:

\begin{lemma}
Let $0<\alpha\leq1$, $1<p<\infty$ and $a(s,t)$ be a non-negative measurable
function on $\left(  -\infty,\infty\right)  \times\left[  0,\infty\right)  $
such that (a.e.)
\begin{align}
a(s,t)  &  \leq1,~\text{when }0<s<t,\label{2.7}\\
\underset{t>0}{\sup}\left(
{\displaystyle\int\limits_{-\infty}^{0}}
+%
{\displaystyle\int\limits_{t}^{\infty}}
a(s,t)^{p^{\prime}}ds\right)  ^{1/p^{\prime}}  &  =b<\infty. \label{2.8}%
\end{align}
Then there is a constant $c_{0}=c_{0}(p,b)$ such that if for $\phi\geq0,$%
\begin{equation}%
{\displaystyle\int\limits_{-\infty}^{\infty}}
\phi(s)^{p}ds\leq1, \label{2.9}%
\end{equation}
then
\begin{equation}%
{\displaystyle\int\limits_{0}^{\infty}}
e^{-F_{\alpha}(t)}dt\leq c_{0} \label{2.10}%
\end{equation}
where
\begin{equation}
F_{\alpha}(t)=\alpha t-\alpha\left(
{\displaystyle\int\limits_{-\infty}^{\infty}}
a(s,t)\phi(s)ds\right)  ^{p^{\prime}}. \label{2.11}%
\end{equation}

\end{lemma}

We sketch a proof here.

\begin{proof}
First, we have
\begin{equation}%
{\displaystyle\int\limits_{0}^{\infty}}
e^{-F_{\alpha}(t)}dt=%
{\displaystyle\int\limits_{-\infty}^{\infty}}
\left\vert E_{\lambda}\right\vert e^{-\lambda}d\lambda. \label{2.12}%
\end{equation}
where $E_{\lambda}=\left\{  t\geq0:F_{\alpha}(t)\leq\lambda\right\}  .$

We will separate the proof into two steps.

\textbf{Step 1:} \, There is a constant $c=c(p,b,\alpha)>0$ such that
$F_{\alpha}(t)\geq-c$ for all $t\geq0$.

Indeed, we will show that if $E_{\alpha\lambda}\neq\emptyset$, then
$\lambda\geq-c$, and furthermore that if $t\in E_{\alpha\lambda}$, then%
\begin{equation}
\left(  b^{p^{\prime}}+t\right)  ^{1/p}\left(
{\displaystyle\int\limits_{t}^{\infty}}
\phi(s)^{p}ds\right)  ^{1/p}\leq A_{1}+B_{1}\left\vert \lambda\right\vert
^{1/p}. \label{2.13}%
\end{equation}

In fact, if $E_{\alpha\lambda}\neq\emptyset$ and $t\in E_{\alpha\lambda}$,
then
\begin{align*}
t-\lambda &  \leq t-\frac{F_{\alpha}(t)}{\alpha}\\
&  \leq\left(
{\displaystyle\int\limits_{-\infty}^{\infty}}
a(s,t)\phi(s)ds\right)  ^{p^{\prime}}%
\end{align*}
Repeating the argument as in the proof of Lemma 1 in \cite{A}, we then have
completed Step 1.

\textbf{Step 2:} \, $\left\vert E_{\lambda}\right\vert \leq A\left\vert
\lambda\right\vert +B$, for constants $A$ and $B$ depending only on $p,~b$ and
$\alpha$.

The proof of Step 2 is very similar to that in \cite{A}. Thus we finish the
proof of the Lemma.
\end{proof}

Using the above lemma, we can provide the

\noindent\textbf{Proof of Theorem 1.1:}

Set $u(x)=I_{n/p}\ast f(x)$, for $f\geq0$. We use the notations
$g(x)=|x|^{\gamma-n}$ and $u^{**}(t)=\frac{1}{t}\int_{0}^{t} u^{*}(s)ds$. Then
by O'Neil's lemma, we have that
\begin{align*}
u^{\ast}(t)  &  \leq u^{\ast\ast}(t)\leq tf^{\ast\ast}(t)g^{\ast\ast}(t)+%
{\displaystyle\int\limits_{0}^{t}}
f^{\ast}(s)g^{\ast}(s)ds\\
&  =\left(  \frac{\omega_{n-1}}{n}\right)  ^{1/p^{\prime}}\left(
pt^{-1/p^{\prime}}%
{\displaystyle\int\limits_{0}^{t}}
f^{\ast}(s)ds+%
{\displaystyle\int\limits_{t}^{\left\vert \Omega\right\vert }}
f^{\ast}(s)s^{-1/p^{\prime}}ds\right)  .
\end{align*}
Now, we change variables by setting $\phi(s)=\left\vert \Omega\right\vert
^{1/p}f^{\ast}(\left\vert \Omega\right\vert e^{-s})e^{-s/p}$, so that%

\begin{align*}
\int_{\Omega}f(x)^{p}dx &  =%
{\displaystyle\int\limits_{0}^{\left\vert \Omega\right\vert }}
f^{\ast}(t)^{p}dt\\
&  =%
{\displaystyle\int\limits_{0}^{\infty}}
\phi(s)^{p}ds.
\end{align*}
By the Hardy-Littlewood inequality, note that with $h(x)=\frac{1}{\left\vert
x\right\vert ^{\alpha}},$ then $h^{\ast}(t)=\left(  \frac{\sigma_{n}}%
{t}\right)  ^{\frac{\alpha}{n}}$, we have
\begin{align*}
&  \int_{\Omega}\frac{\exp\left(  \left(  1-\frac{\alpha}{n}\right)  \frac
{n}{\omega_{n-1}}\left\vert u(x)\right\vert ^{p^{\prime}}\right)  }{\left\vert
x\right\vert ^{\alpha}}dx\\
&  \leq\sigma_{n}^{\frac{\alpha}{n}}%
{\displaystyle\int\limits_{0}^{\left\vert \Omega\right\vert }}
\frac{e^{\left(  1-\frac{\alpha}{n}\right)  \frac{n}{\omega_{n-1}}u^{\ast
}(t)^{p^{\prime}}}}{t^{\frac{\alpha}{n}}}\\
&  =\sigma_{n}^{\frac{\alpha}{n}}\left\vert \Omega\right\vert ^{1-\frac
{\alpha}{n}}%
{\displaystyle\int\limits_{0}^{\infty}}
\exp\left[  \left(  1-\frac{\alpha}{n}\right)  \frac{n}{\omega_{n-1}}u^{\ast
}\left(  \left\vert \Omega\right\vert e^{-s}\right)  ^{p^{\prime}}-\left(
1-\frac{\alpha}{n}\right)  s\right]  ds\\
&  \leq\sigma_{n}^{\frac{\alpha}{n}}\left\vert \Omega\right\vert
^{1-\frac{\alpha}{n}}\times\\
&
{\displaystyle\int\limits_{0}^{\infty}}
\exp\left[  \left(  1-\frac{\alpha}{n}\right)  \left(  p\left(  \left\vert
\Omega\right\vert e^{-s}\right)  ^{-\frac{1}{p^{\prime}}}%
{\displaystyle\int\limits_{0}^{\left\vert \Omega\right\vert e^{-s}}}
f^{\ast}(z)dz+%
{\displaystyle\int\limits_{\left\vert \Omega\right\vert e^{-s}}^{\left\vert
\Omega\right\vert }}
f^{\ast}(z)z^{-\frac{1}{p^{\prime}}}dz\right)  ^{p^{\prime}}-\left(
1-\frac{\alpha}{n}\right)  s\right]  ds\\
&  =\sigma_{n}^{\frac{\alpha}{n}}\left\vert \Omega\right\vert ^{1-\frac
{\alpha}{n}}%
{\displaystyle\int\limits_{0}^{\infty}}
\exp\left[  \left(  1-\frac{\alpha}{n}\right)  \left(  pe^{s/p^{\prime}}%
{\displaystyle\int\limits_{s}^{\infty}}
\phi(w)e^{-\frac{w}{p^{\prime}}}dw+%
{\displaystyle\int\limits_{0}^{s}}
\phi(w)\right)  ^{p^{\prime}}-\left(  1-\frac{\alpha}{n}\right)  s\right]
ds\\
&  =\sigma_{n}^{\frac{\alpha}{n}}\left\vert \Omega\right\vert ^{1-\frac
{\alpha}{n}}%
{\displaystyle\int\limits_{0}^{\infty}}
\exp\left[  -F_{\left(  1-\frac{\alpha}{n}\right)  }(s)\right]  ds.
\end{align*}
where $F_{\left(  1-\frac{\alpha}{n}\right)  }(s)$ is as in Lemma 3.1 with
\[
a(s,t)=\left\{
\begin{array}
[c]{c}%
1\,\,\text{ for }0<s<t\\
pe^{(t-s)/p^{\prime}}\,\,\text{ for }t<s<\infty\\
0\,\,\text{ for }-\infty<s\leq0
\end{array}
\right.  .
\]
Thus it suffices to prove that
\[%
{\displaystyle\int\limits_{0}^{\infty}}
\phi(s)^{p}ds\leq1\text{ implies }%
{\displaystyle\int\limits_{0}^{\infty}}
\exp\left[  -F_{\left(  1-\frac{\alpha}{n}\right)  }(s)\right]  ds\leq c_{0},
\]
but this follows from Lemma 3.1 immediately.

\section{\bigskip Proof of Theorem 1.2: A singular Adams inequality on bounded
domains}

First, we will prove that
\begin{equation}
\label{inequality}\underset{u\in W_{0}^{m,\frac{n}{m}}\left(  \Omega\right)
,~\left\Vert \nabla^{m}u\right\Vert _{\frac{n}{m}}\leq1}{\sup}\int_{\Omega
}\frac{e^{\beta_{\alpha,n,m}\left\vert u\right\vert ^{\frac{n}{n-m}}}%
}{\left\vert x\right\vert ^{\alpha}}dx<\infty.
\end{equation}
To do that, it suffices to dominate an arbitrary $C^{m}$ function with compact
support by a Riesz potential in such a way that the constants are precise.
This can be done as in \cite{A} through the following lemma:

\begin{lemma}
\label{lemma2.1} Let $u\in C_{0}^{\infty}\left(
\mathbb{R}
^{n}\right)  $. Set $p=\frac{n}{m}$ and $p^{\prime}=\frac{n}{n-m}$. Then if
$m$ is an odd positive integer,%
\[
u(x)=(-1)^{\frac{m-1}{2}}\left(  \frac{\omega_{n-1}\beta(n, m)}{n}\right)
^{-1/p^{\prime}}\times\int_{%
\mathbb{R}
^{n}}\left\vert x-y\right\vert ^{m-1-n}(x-y)\cdot\nabla^{m}u(y)dy
\]
and for $m$ an even positive integer%
\[
u(x)=(-1)^{\frac{m}{2}}\left(  \frac{\omega_{n-1}\beta(n, m)}{n}\right)
^{-1/p^{\prime}}\times\int_{%
\mathbb{R}
^{n}}\left\vert x-y\right\vert ^{m-n}\nabla^{m}u(y)dy
\]

\end{lemma}

\noindent\textbf{Proof of Theorem 1.2:}

It is clear that from Lemma \ref{lemma2.1}, we have $\left(  \frac
{\omega_{n-1}}{n}\right)  \beta(n, m)\left\vert u(x)\right\vert ^{p^{\prime}%
}\leq\left[  I_{m}\ast\left\vert \nabla^{m}u\right\vert (x)\right]
^{p^{\prime}}$ and then we apply Theorem 1.1. This proves the first part of
Theorem 1.2.

To show the second part of Theorem 1.2. Now, suppose $m$ is even: $m=2k,~k\in%
\mathbb{N}
$, we will prove that
\begin{equation}
\underset{u\in W_{N}^{m,\frac{n}{m}}\left(  \Omega\right)  ,~\left\Vert
\nabla^{m}u\right\Vert _{\frac{n}{m}}\leq1}{\sup}\int_{\Omega}\frac
{e^{\beta_{\alpha,n,m}\left\vert u\right\vert ^{\frac{n}{n-m}}}}{\left\vert
x\right\vert ^{\alpha}}dx<\infty\label{3.0}%
\end{equation}
By a density argument, it is enough to prove that
\[
\underset{u\in C_{N}^{\infty}\left(  \Omega\right)  ,~\left\Vert \nabla
^{m}u\right\Vert _{\frac{n}{m}}\leq1}{\sup}\int_{\Omega}\frac{e^{\beta
_{\alpha,n,m}\left\vert u\right\vert ^{\frac{n}{n-m}}}}{\left\vert
x\right\vert ^{\alpha}}dx<\infty
\]
where
\[
C_{N}^{\infty}\left(  \Omega\right)  =\left\{  u\in C^{\infty}\left(
\Omega\right)  \cap C^{m-2}\left(  \overline{\Omega}\right)  :u|_{\partial
\Omega}=\Delta^{j}u|_{\partial\Omega}=0,~1\leq j\leq\left[  \frac{m-1}%
{2}\right]  \right\}  .
\]

Let $u\in C_{N}^{\infty}\left(  \Omega\right)  $ be such that $\left\Vert
\nabla^{m}u\right\Vert _{\frac{n}{m}}=\left\Vert \Delta^{k}u\right\Vert
_{\frac{n}{m}}\leq1$ and set $f:=\Delta^{k}u$ in $\Omega$. Then $u$ is a
solution of the Navier boundary value problem%
\[
\left\{
\begin{array}
[c]{c}%
\Delta^{k}u=f\text{ in }\Omega\\
u=\Delta^{j}u=0\text{ on }\partial\Omega,~j\in\left\{  1\leq j<k\right\}
\end{array}
\right.  .
\]
Now, we extend $f$ by zero outside $\Omega$%
\[
\overline{f}(x)=\left\{
\begin{array}
[c]{c}%
f(x),\text{ }x\in\Omega\\
0,~x\in%
\mathbb{R}
^{n}\setminus\Omega
\end{array}
\right.  .
\]
Define
\[
\overline{u}=\left(  \frac{n}{\omega_{n-1}\beta\left(  n,m\right)  }\right)
^{\frac{n-m}{n}}I_{m}\ast\left\vert \overline{f}\right\vert \text{ in }%
\mathbb{R}
^{n}\text{,}%
\]
so that we have $\left(  -1\right)  ^{k}\Delta^{k}u=\left\vert \overline
{f}\right\vert $ in $%
\mathbb{R}
^{n}$. It's clear that $\overline{u}\geq0$ in $%
\mathbb{R}
^{n}$ and
\[
\beta\left(  n,m\right)  \left\vert \overline{u}\right\vert ^{\frac{n}{n-m}%
}\leq\frac{n}{\omega_{n-1}}\left(  \frac{I_{m}\ast\left\vert \overline
{f}\right\vert }{\left\Vert f\right\Vert _{\frac{n}{m}}}\right)  ^{\frac
{n}{n-m}}\text{ in }%
\mathbb{R}
^{n}.
\]
It can be proved that $\overline{u}\geq\left\vert u\right\vert $ (see
\cite{RS}) and then
\begin{align*}
\int_{\Omega}\frac{e^{\beta_{\alpha,n,m}\left\vert u\right\vert ^{\frac
{n}{n-m}}}}{\left\vert x\right\vert ^{\alpha}}dx  &  \leq\int_{\Omega}%
\frac{e^{\beta_{\alpha,n,m}\left\vert \overline{u}\right\vert ^{\frac{n}{n-m}%
}}}{\left\vert x\right\vert ^{\alpha}}dx\\
&  \leq\int_{\Omega}\frac{\exp\left(  \left(  1-\frac{\alpha}{n}\right)
\frac{n}{\omega_{n-1}}\left\vert \frac{I_{\beta}\ast\overline{f}%
(x)}{\left\Vert \overline{f}\right\Vert _{\frac{n}{m}}}\right\vert ^{\frac
{n}{n-m}}\right)  }{\left\vert x\right\vert ^{\alpha}}dx
\end{align*}
By Theorem 1.1, (\ref{3.0}) follows.

Moreover, it can be checked that the sequence of test functions which gives
the sharpness of Adams' inequality in bounded domains \cite{A} also gives the
sharpness of $\beta_{\alpha,n,m}$. This completes the proof of Theorem 1.2.

\section{Proof of Theorem 1.3 and Theorem 1.4}

\subsection{Proof of Theorem 1.3}

\begin{proof}
Suppose that $m=2k,~k\in%
\mathbb{N}
$. Let $u\in W^{m,\frac{n}{m}}\left(
\mathbb{R}
^{n}\right)  ,\left\Vert \left(  -\Delta+I\right)  ^{k}u\right\Vert _{\frac
{n}{m}}\leq1$, by the density of $C_{0}^{\infty}\left(
\mathbb{R}
^{n}\right)  $ in $W^{m,\frac{n}{m}}\left(
\mathbb{R}
^{n}\right)  $, without loss of generality, we can find a sequence of
functions $u_{l}\in C_{0}^{\infty}\left(
\mathbb{R}
^{n}\right)  $ such that $u_{l}\rightarrow u$ in $W^{m,\frac{n}{m}}\left(
\mathbb{R}
^{n}\right)  $ and $\int_{%
\mathbb{R}
^{n}}\left\vert \left(  -\Delta+I\right)  ^{k}u_{l}\right\vert ^{\frac{n}{m}%
}dx\leq1$ and suppose that supp$u_{l}\subset B_{R_{l}}$ for any fixed $l$. Let
$f_{l}:=\left(  -\Delta+I\right)  ^{k}u_{l}$, then supp$f_{l}\subset B_{R_{l}%
}$. Consider the problem%
\[
\left\{
\begin{array}
[c]{c}%
\left(  -\Delta+I\right)  ^{k}v_{l}=f_{l}^{\#}\\
v_{l}\in W_{N}^{m,2}\left(  B_{R_{l}}\right)
\end{array}
\right.  .
\]
By the property of rearrangement, we have
\begin{equation}
\int_{B_{R_{l}}}\left\vert \left(  -\Delta+I\right)  ^{k}v_{l}\right\vert
^{\frac{n}{m}}dx=\int_{B_{R_{l}}}\left\vert \left(  -\Delta+I\right)
^{k}u_{l}\right\vert ^{\frac{n}{m}}dx\leq1 \label{3.1}%
\end{equation}
and by the Hardy-Littlewood inequality and Proposition 2.4, we get
\[
\int_{B_{R_{l}}}\frac{\phi\left(  \beta_{\alpha,n,m}\left\vert u_{l}%
\right\vert ^{\frac{n}{n-m}}\right)  }{\left\vert x\right\vert ^{\alpha}%
}dx\leq\int_{B_{R_{l}}}\frac{\phi\left(  \beta_{\alpha,n,m}\left\vert
u_{l}^{\#}\right\vert ^{\frac{n}{n-m}}\right)  }{\left\vert x\right\vert
^{\alpha}}dx\leq\int_{B_{R_{l}}}\frac{\phi\left(  \beta_{\alpha,n,m}\left\vert
v_{l}\right\vert ^{\frac{n}{n-m}}\right)  }{\left\vert x\right\vert ^{\alpha}%
}dx
\]
Now, writing
\begin{align*}
\int_{B_{R_{l}}}\frac{\phi\left(  \beta_{\alpha,n,m}\left\vert v_{l}%
\right\vert ^{\frac{n}{n-m}}\right)  }{\left\vert x\right\vert ^{\alpha}}dx
&  \leq\int_{B_{R_{0}}}\frac{\phi\left(  \beta_{\alpha,n,m}\left\vert
v_{l}\right\vert ^{\frac{n}{n-m}}\right)  }{\left\vert x\right\vert ^{\alpha}%
}dx+\int_{B_{R_{l}}\smallsetminus B_{R_{0}}}\frac{\phi\left(  \beta
_{\alpha,n,m}\left\vert v_{l}\right\vert ^{\frac{n}{n-m}}\right)  }{\left\vert
x\right\vert ^{\alpha}}dx\\
&  =I_{1}+I_{2}%
\end{align*}
where $R_{0}$ is a constant and will be chosen later. Then we will prove that
both $I_{1}$ and $I_{2}$ are bounded uniformly by a constant.

Using Theorem 1.2, we can estimate $I_{1}$. Indeed, we just need to construct
an auxiliary radial function $w_{l}\in W_{N}^{m,\frac{n}{m}}\left(  B_{R_{0}%
}\right)  $ with $\left\Vert \nabla^{m}w_{l}\right\Vert _{\frac{n}{m}}\leq1$
which increases the integral we are interested in. Such a function was
constructed in \cite{RS}. For the sake of completion, we give the detail here.
For each $i\in\left\{  1,2,...,m-1\right\}  $ we define
\[
g_{i}\left(  \left\vert x\right\vert \right)  :=\left\vert x\right\vert
^{m-i},~\forall x\in B_{R_{0}}%
\]
so $g_{i}\in W_{rad}^{m,\frac{n}{m}}\left(  B_{R_{0}}\right)  $. Moreover,
\[
\Delta^{j}g_{i}\left(  \left\vert x\right\vert \right)  =\left\{
\begin{array}
[c]{c}%
c_{i}^{j}\left\vert x\right\vert ^{m-i-2j}\text{ for }j\in\left\{
1,2,...k-i\right\} \\
0\text{ for }j\in\left\{  k-i+1,...,k\right\}
\end{array}
\right.  ~\forall x\in B_{R_{0}}%
\]
where
\[
c_{i}^{j}=%
{\displaystyle\prod\limits_{h=1}^{j}}
\left[  n+m-2\left(  h+i\right)  \right]  \left[  m-2\left(  i+h-1\right)
\right]  \text{, }\forall j\in\left\{  1,2,...k-i\right\}  .
\]
Let
\[
z_{l}\left(  \left\vert x\right\vert \right)  :=v_{l}\left(  \left\vert
x\right\vert \right)  -%
{\displaystyle\sum\limits_{i=1}^{k-1}}
a_{i}g_{i}\left(  \left\vert x\right\vert \right)  -a_{k}~\forall x\in
B_{R_{0}}%
\]
where
\begin{align*}
a_{i}  &  :=\frac{\Delta^{k-i}v_{l}\left(  R_{0}\right)  -%
{\displaystyle\sum\limits_{j=1}^{i-1}}
a_{j}\Delta^{k-i}g_{j}\left(  R_{0}\right)  }{\Delta^{k-i}g_{i}\left(
R_{0}\right)  },~\forall i\in\left\{  1,2,...k-1\right\}  ,\\
a_{k}  &  :=v_{l}\left(  R_{0}\right)  -%
{\displaystyle\sum\limits_{i=1}^{k-1}}
a_{i}g_{i}\left(  R_{0}\right)  .
\end{align*}
We can check that (see \cite{RS})
\begin{align*}
z_{l}  &  \in W_{N,rad}^{m,\frac{n}{m}}\left(  B_{R_{0}}\right)  ,\\
\nabla^{m}v_{l}  &  =\nabla^{m}z_{l}\text{ in }B_{R_{0}}\text{.}%
\end{align*}
We have the following lemma whose proof can be found in \cite{RS}.

\begin{lemma}
For $0<\left\vert x\right\vert \leq R_{0}$, there exists some constant
$d(m,n,R_{0})$ depending only on $m,n,R_{0}$ such that
\begin{align*}
\left\vert v_{l}\left(  \left\vert x\right\vert \right)  \right\vert
^{\frac{n}{n-m}}  & \leq\left\vert z_{l}\left(  \left\vert x\right\vert
\right)  \right\vert ^{\frac{n}{n-m}}\left(  1+c_{m,n}%
{\displaystyle\sum\limits_{j=1}^{k-1}}
\frac{1}{R_{0}^{2j\frac{n}{m}-1}}\left\Vert \Delta^{k-j}v_{l}\right\Vert
_{W^{1,\frac{n}{m}}}^{\frac{n}{m}}+\frac{c_{m,n}}{R_{0}^{n-1}}\left\Vert
v_{l}\right\Vert _{W^{1,\frac{n}{m}}}^{\frac{n}{m}}\right)  ^{\frac{n}{n-m}%
}\\
& +d(m,n,R_{0}).
\end{align*}

\end{lemma}

Now, setting
\[
w_{l}\left(  \left\vert x\right\vert \right)  :=z_{l}\left(  \left\vert
x\right\vert \right)  \left(  1+c_{m,n}%
{\displaystyle\sum\limits_{j=1}^{k-1}}
\frac{1}{R_{0}^{2j\frac{n}{m}-1}}\left\Vert \Delta^{k-j}v_{l}\right\Vert
_{W^{1,\frac{n}{m}}}^{\frac{n}{m}}+\frac{c_{m,n}}{R_{0}^{n-1}}\left\Vert
v_{l}\right\Vert _{W^{1,\frac{n}{m}}}^{\frac{n}{m}}\right)  .
\]
Since
\begin{align*}
z_{l}  &  \in W_{N,rad}^{m,\frac{n}{m}}\left(  B_{R_{0}}\right)  ,\\
\nabla^{m}v_{l}  &  =\nabla^{m}z_{l}\text{ in }B_{R_{0}}\text{.}%
\end{align*}
we have
\[
w_{l}\in W_{N,rad}^{m,\frac{n}{m}}\left(  B_{R_{0}}\right)
\]
and%
\[
\left\Vert \nabla^{m}w_{l}\right\Vert _{\frac{n}{m}}=\left\Vert \nabla
^{m}z_{l}\right\Vert _{\frac{n}{m}}\left(  1+c_{m,n}%
{\displaystyle\sum\limits_{j=1}^{k-1}}
\frac{1}{R_{0}^{2j\frac{n}{m}-1}}\left\Vert \Delta^{k-j}v_{l}\right\Vert
_{W^{1,\frac{n}{m}}}^{\frac{n}{m}}+\frac{c_{m,n}}{R_{0}^{n-1}}\left\Vert
v_{l}\right\Vert _{W^{1,\frac{n}{m}}}^{\frac{n}{m}}\right)  .
\]
Note that
\begin{align*}
\left\Vert \nabla^{m}z_{l}\right\Vert _{\frac{n}{m}}  &  =\left\Vert
\nabla^{m}v_{l}\right\Vert _{\frac{n}{m}}\\
&  \leq\left(  1-%
{\displaystyle\sum\limits_{j=1}^{k-1}}
\left\Vert \Delta^{k-j}v_{l}\right\Vert _{W^{1,\frac{n}{m}}}^{\frac{n}{m}%
}-\left\Vert v_{l}\right\Vert _{W^{1,\frac{n}{m}}}^{\frac{n}{m}}\right)
^{m/n}\\
&  \leq1-\frac{m}{n}%
{\displaystyle\sum\limits_{j=1}^{k-1}}
\left\Vert \Delta^{k-j}v_{l}\right\Vert _{W^{1,\frac{n}{m}}}^{\frac{n}{m}%
}-\frac{m}{n}\left\Vert v_{l}\right\Vert _{W^{1,\frac{n}{m}}}^{\frac{n}{m}}%
\end{align*}
we have%
\begin{align*}
\left\Vert \nabla^{m}w_{l}\right\Vert _{\frac{n}{m}}  &  \leq\left(
1-\frac{m}{n}%
{\displaystyle\sum\limits_{j=1}^{k-1}}
\left\Vert \Delta^{k-j}v_{l}\right\Vert _{W^{1,\frac{n}{m}}}^{\frac{n}{m}%
}-\frac{m}{n}\left\Vert v_{l}\right\Vert _{W^{1,\frac{n}{m}}}^{\frac{n}{m}%
}\right)  \times\\
&  \times\left(  1+c_{m,n}%
{\displaystyle\sum\limits_{j=1}^{k-1}}
\frac{1}{R_{0}^{2j\frac{n}{m}-1}}\left\Vert \Delta^{k-j}v_{l}\right\Vert
_{W^{1,\frac{n}{m}}}^{\frac{n}{m}}+\frac{c_{m,n}}{R_{0}^{n-1}}\left\Vert
v_{l}\right\Vert _{W^{1,\frac{n}{m}}}^{\frac{n}{m}}\right) \\
&  \leq1
\end{align*}
if we choose $R_{0}$ sufficiently large.

Finally, note that
\[
I_{1}\leq e^{\beta_{0}d(m,n,R_{0})}\int_{B_{R_{0}}}\frac{e^{\beta_{0}w_{l}%
^{2}}}{\left\vert x\right\vert ^{\alpha}}dx,
\]
using Theorem 1.2, we can conclude that $I_{1}$ is bounded by a constant since
$w_{l}\in W_{N,rad}^{m,\frac{n}{m}}\left(  B_{R_{0}}\right)  $ and $\left\Vert
\nabla^{m}w_{l}\right\Vert _{\frac{n}{m}}\leq1.$

Now, we will estimate $I_{2}$. Note that%
\begin{align*}
I_{2}  &  =\int_{B_{R_{l}}\smallsetminus B_{R_{0}}}\frac{\phi\left(
\beta_{\alpha,n,m}\left\vert v_{l}\right\vert ^{\frac{n}{n-m}}\right)
}{\left\vert x\right\vert ^{\alpha}}dx\\
&  \leq\frac{1}{R_{0}^{\alpha}}\int_{B_{R_{l}}\smallsetminus B_{R_{0}}}%
\phi\left(  \beta_{\alpha,n,m}\left\vert v_{l}\right\vert ^{\frac{n}{n-m}%
}\right)  dx
\end{align*}
By the same argument as that in \cite{RS}, we can conclude that $I_{2}\leq
c\left(  m,n,R_{0}\right)  .$

Combining the above estimates and using the Fatou lemma, we can conclude that
\[
\underset{u\in W^{m,\frac{n}{m}}\left(
\mathbb{R}
^{n}\right)  ,\left\Vert \left(  -\Delta+I\right)  ^{\frac{m}{2}}u\right\Vert
_{\frac{n}{m}}\leq1}{\sup}\int_{%
\mathbb{R}
^{n}}\frac{\phi\left(  \beta_{\alpha,n,m}\left\vert u\right\vert ^{\frac
{n}{n-m}}\right)  dx}{\left\vert x\right\vert ^{\alpha}}dx<\infty.
\]

When $\beta>\beta_{\alpha,n,m}$, again, it's easy to check that the sequence
given by D. Adams \cite{A} will make our supremum blow up. This completes the
proof of Theorem 1.3.
\end{proof}

\subsection{Proof of Theorem 1.4}

\begin{proof}
It suffices to prove that
\[
\underset{u\in W^{2,2}\left(
\mathbb{R}
^{4}\right)  ,\int_{%
\mathbb{R}
^{4}}\left(  \left\vert \Delta u\right\vert ^{2}+\tau\left\vert \nabla
u\right\vert ^{2}+\sigma\left\vert u\right\vert ^{2}\right)  \leq1}{\sup}%
\int_{%
\mathbb{R}
^{4}}\frac{\left(  e^{32\pi^{2}\left(  1-\frac{\alpha}{4}\right)  u^{2}%
}-1\right)  }{\left\vert x\right\vert ^{\alpha}}dx<\infty.
\]
In fact, we will prove a stronger result that
\begin{equation}
\underset{u\in W^{2,2}\left(
\mathbb{R}
^{4}\right)  ,\left\Vert -\Delta u+cu\right\Vert _{2}\leq1}{\sup}\int_{%
\mathbb{R}
^{4}}\frac{\left(  e^{32\pi^{2}\left(  1-\frac{\alpha}{4}\right)  u^{2}%
}-1\right)  }{\left\vert x\right\vert ^{\alpha}}dx<\infty\label{4.0}%
\end{equation}
where $c>0$ is chosen such that $\left\Vert -\Delta u+cu\right\Vert _{2}%
^{2}\leq\int_{%
\mathbb{R}
^{4}}\left(  \left\vert \Delta u\right\vert ^{2}+\tau\left\vert \nabla
u\right\vert ^{2}+\sigma\left\vert u\right\vert ^{2}\right)  .$

Let $u\in W^{2,2}\left(
\mathbb{R}
^{4}\right)  ,\left\Vert -\Delta u+cu\right\Vert _{2}\leq1$. By the density of
$C_{0}^{\infty}\left(
\mathbb{R}
^{4}\right)  $ in $W^{2,2}\left(
\mathbb{R}
^{4}\right)  $, we can find a sequence of functions $u_{k}$ in $C_{0}^{\infty
}\left(
\mathbb{R}
^{4}\right)  $ such that $u_{k}\rightarrow u$ in $W^{2,2}\left(
\mathbb{R}
^{4}\right)  ,$ supp\, $u\subset B_{R_{k}}$. Without loss of generality, we
assume $\left\Vert -\Delta u_{k}+cu_{k}\right\Vert _{2}\leq1.$ By the Fatou
lemma, we have
\begin{equation}
\int_{%
\mathbb{R}
^{4}}\frac{\left(  e^{32\pi^{2}\left(  1-\frac{\alpha}{4}\right)  u^{2}%
}-1\right)  }{\left\vert x\right\vert ^{\alpha}}dx\leq\underset{k\rightarrow
\infty}{\lim\inf}\int_{B_{R_{k}}}\frac{\left(  e^{32\pi^{2}\left(
1-\frac{\alpha}{4}\right)  u_{k}^{2}}-1\right)  }{\left\vert x\right\vert
^{\alpha}}dx. \label{4.1}%
\end{equation}
Now, set $f_{k}:=-\Delta u_{k}+cu_{k}$ and consider the problem
\[
\left\{
\begin{array}
[c]{c}%
-\Delta v_{k}+cv_{k}=f_{k}^{\#}\text{ in }B_{R_{k}}\\
v_{k}\in W_{0}^{1,2}\left(  B_{R_{k}}\right)
\end{array}
\right.  .
\]
We have that $v_{k}\in W_{N}^{2,2}\left(  B_{R_{k}}\right)  $. Moreover, by
Proposition 2.4 and the property of rearrangement, we have
\begin{align}
\left\Vert -\Delta u_{k}+cu_{k}\right\Vert _{2}  &  =\left\Vert -\Delta
v_{k}+cv_{k}\right\Vert _{2}\leq1\label{4.2}\\
\int_{B_{R_{k}}}\frac{\left(  e^{32\pi^{2}\left(  1-\frac{\alpha}{4}\right)
u_{k}^{2}}-1\right)  }{\left\vert x\right\vert ^{\alpha}}dx  &  \leq
\int_{B_{R_{k}}}\frac{\left(  e^{32\pi^{2}\left(  1-\frac{\alpha}{4}\right)
v_{k}^{2}}-1\right)  }{\left\vert x\right\vert ^{\alpha}}dx.\nonumber
\end{align}
Now, we write
\begin{align*}
&  \int_{B_{R_{k}}}\frac{\left(  e^{32\pi^{2}\left(  1-\frac{\alpha}%
{4}\right)  v_{k}^{2}}-1\right)  }{\left\vert x\right\vert ^{\alpha}}dx\\
=  &  \int_{B_{R_{0}}}\frac{\left(  e^{32\pi^{2}\left(  1-\frac{\alpha}%
{4}\right)  v_{k}^{2}}-1\right)  }{\left\vert x\right\vert ^{\alpha}}%
dx+\int_{B_{R_{k}}\setminus B_{R_{0}}}\frac{\left(  e^{32\pi^{2}\left(
1-\frac{\alpha}{4}\right)  v_{k}^{2}}-1\right)  }{\left\vert x\right\vert
^{\alpha}}dx\\
&  =I_{1}+I_{2}.
\end{align*}
where $R_{0}$ only depends on $c$ and will be chosen later.

Choose $R_{0}\geq\left(  \frac{1}{2\pi^{2}}\left(  \frac{1}{2c}+\frac{1}%
{c^{2}}\right)  \right)  ^{1/3}$, then by the Radial Lemma (Lemma 2.2) and
(\ref{4.2}), we have that $\left\vert v_{k}(x)\right\vert \leq1$ when
$\left\vert x\right\vert \geq R_{0}.$ Thus
\begin{align}
I_{2}  &  =\int_{B_{R_{k}}\setminus B_{R_{0}}}\frac{\left(  e^{32\pi
^{2}\left(  1-\frac{\alpha}{4}\right)  v_{k}^{2}}-1\right)  }{\left\vert
x\right\vert ^{\alpha}}dx\label{4.3}\\
&  \leq\frac{1}{R_{0}^{\alpha}}\int_{B_{R_{k}}\setminus B_{R_{0}}}\left(
e^{32\pi^{2}\left(  1-\frac{\alpha}{4}\right)  v_{k}^{2}}-1\right)
dx\nonumber\\
&  \leq\frac{1}{R_{0}^{\alpha}}%
{\displaystyle\sum\limits_{j=1}^{\infty}}
\frac{\left(  32\pi^{2}\left(  1-\frac{\alpha}{4}\right)  \right)  ^{j}}%
{j!}\int_{B_{R_{k}}}v_{k}^{2}\nonumber\\
&  \leq\frac{1}{R_{0}^{\alpha}}\frac{1}{c^{2}}%
{\displaystyle\sum\limits_{j=1}^{\infty}}
\frac{\left(  32\pi^{2}\left(  1-\frac{\alpha}{4}\right)  \right)  ^{j}}%
{j!}\nonumber\\
&  =C(c).\nonumber
\end{align}
Now, we estimate $I_{1}$. Put%
\[
w_{k}(\left\vert x\right\vert )=\left\{
\begin{array}
[c]{c}%
v_{k}(\left\vert x\right\vert )-v_{k}\left(  R_{0}\right)  \text{, }%
0\leq\left\vert x\right\vert \leq R_{0}\\
0\text{ , }r\geq R_{0}%
\end{array}
\right.  .
\]
Then it's easy to check that $w_{k}\in W_{N}^{2,2}\left(  B_{R_{0}}\right)  $.
Moreover, when $0<\left\vert x\right\vert \leq R_{0}$, using Radial Lemmas
(Lemma 2.2 and 2.3), we have
\begin{align*}
\left(  v_{k}\left(  \left\vert x\right\vert \right)  \right)  ^{2}  &
=\left[  w_{k}\left(  \left\vert x\right\vert \right)  +v_{k}\left(
R_{0}\right)  \right]  ^{2}\\
&  =w_{k}^{2}\left(  \left\vert x\right\vert \right)  +2w_{k}\left(
\left\vert x\right\vert \right)  v_{k}\left(  R_{0}\right)  +\left[
v_{k}\left(  R_{0}\right)  \right]  ^{2}\\
&  \leq w_{k}^{2}\left(  \left\vert x\right\vert \right)  +w_{k}^{2}\left(
\left\vert x\right\vert \right)  \left[  v_{k}\left(  R_{0}\right)  \right]
^{2}+1+\left[  v_{k}\left(  R_{0}\right)  \right]  ^{2}\\
&  \leq w_{k}^{2}\left(  \left\vert x\right\vert \right)  \left[  1+\frac
{C}{R_{0}^{2}}\left\Vert v_{k}\right\Vert _{W^{1,2}}^{2}\right]  +d(c,R_{0}).
\end{align*}
Let
\[
z_{k}(\left\vert x\right\vert ):=w_{k}\left(  \left\vert x\right\vert \right)
\left[  1+\frac{C}{R_{0}^{2}}\left\Vert v_{k}\right\Vert _{W^{1,2}}%
^{2}\right]  ^{1/2}%
\]
then $z_{k}\in W_{N}^{2,2}\left(  B_{R_{0}}\right)  $ since $w_{k}\in
W_{N}^{2,2}\left(  B_{R_{0}}\right)  $. More importantly, we have
\begin{align*}
\left\Vert \Delta z_{k}\right\Vert _{2}^{2}  &  =\left\Vert \Delta
w_{k}\right\Vert _{2}^{2}\left[  1+\frac{C}{R_{0}^{2}}\left\Vert
v_{k}\right\Vert _{W^{1,2}}^{2}\right] \\
&  =\left\Vert \Delta v_{k}\right\Vert _{2}^{2}\left[  1+\frac{C}{R_{0}^{2}%
}\left\Vert v_{k}\right\Vert _{W^{1,2}}^{2}\right] \\
&  \leq(1-2c\left\Vert \nabla v_{k}\right\Vert _{2}^{2}-c^{2}\left\Vert
v_{k}\right\Vert _{2}^{2})(1+\frac{C}{R_{0}^{2}}\left\Vert \nabla
v_{k}\right\Vert _{2}^{2}+\frac{C}{R_{0}^{2}}\left\Vert v_{k}\right\Vert
_{2}^{2})\\
&  \leq1
\end{align*}
if we choose $R_{0}$ sufficiently large.

Then using Theorem 1.2, we have
\begin{align}
I_{1}  &  =\int_{B_{R_{0}}}\frac{\left(  e^{32\pi^{2}\left(  1-\frac{\alpha
}{4}\right)  v_{k}^{2}}-1\right)  }{\left\vert x\right\vert ^{\alpha}%
}dx\label{4.4}\\
&  \leq C\left(  c\right)  \int_{B_{R_{0}}}\frac{e^{32\pi^{2}\left(
1-\frac{\alpha}{4}\right)  z_{k}^{2}}}{\left\vert x\right\vert ^{\alpha}%
}dx\nonumber\\
&  \leq C(c).\nonumber
\end{align}
From (\ref{4.3}) and (\ref{4.4}), we get (\ref{4.0}). Moreover, if we choose
$0<c<\min\left\{  \frac{\tau}{2},\sqrt{\sigma}\right\}  $, then we have
\[
\left\Vert -\Delta u+cu\right\Vert _{2}^{2}\leq\int_{%
\mathbb{R}
^{4}}\left(  \left\vert \Delta u\right\vert ^{2}+\tau\left\vert \nabla
u\right\vert ^{2}+\sigma\left\vert u\right\vert ^{2}\right)  ,\ \forall u\in
W^{2,2}\left(
\mathbb{R}
^{4}\right)
\]
and thus the proof of Theorem 1.4 is completed.
\end{proof}

\textbf{Acknowledgement:} The authors wish to thank the local organizers of
the International Conference in Geometry, Analysis and PDEs at Jiaxing, China
where this work was presented.

\end{document}